\documentclass[10pt,a4paper,oneside]{article}
\usepackage[english]{babel}
\usepackage{a4wide,amsmath,amsthm,amssymb,url,graphicx,xspace,algorithm,algorithmic,pgf}
\usepackage[latin1]{inputenc}

\def\F{\mathbb{F}}
\def\N{\mathbb{N}}

\DeclareMathOperator{\supp}{supp}
\DeclareMathOperator{\wt}{wt}
\DeclareMathOperator{\PG}{PG}

\DeclareMathOperator{\Aut}{Aut}

\theoremstyle{definition}
\newtheorem{theorem}{Theorem}[section]
\newtheorem{lemma}[theorem]{Lemma}
\newtheorem{definition}[theorem]{Definition}
\newtheorem{remark}[theorem]{Remark}

\newtheorem{notation}[theorem]{Notation}
\newtheorem{example}[theorem]{Example}

\newcommand{\comments}[1]{}

\author{M. De Boeck\footnote{The research of this author is supported by FWO-Vlaanderen (Research Foundation - Flanders). \newline Address: UGent, Department of Mathematics, Krijgslaan 281-S22, 9000 Gent, Flanders, Belgium. \newline Email address: \{mdeboeck,pv\}@cage.ugent.be.} { }and P. Vandendriessche\footnotemark[\value{footnote}]}
\title{On the dual code of points and generators on the Hermitian variety $\mathcal{H}(2n+1,q^{2})$}
\begin{document}
\maketitle

\begin{abstract}
  We study the dual linear code of points and generators on a non-singular Hermitian variety $\mathcal{H}(2n+1,q^2)$. We improve the earlier results for $n=2$, we solve the minimum distance problem for general $n$, we classify the $n$ smallest types of code words and we characterize all small weight code words as being a linear combination of these $n$ types.
\end{abstract}

\section{Preliminaries}

Over the last decades, several types of linear codes arising from finite geometries have been studied. The general concept is as follows: one considers the point set $\mathcal P$ of a geometric space or variety $\mathcal S$, together with a collection $\mathcal B$ of $k$-dimensional subspaces on $\mathcal S$. 

\begin{definition}
  The incidence matrix $A\in\{0,1\}^{\mathcal B\times\mathcal P}$ is the unique matrix such that $H_{b,p}=1$ if $p\in b$ and $H_{b,p}=0$ otherwise.
\end{definition}

Since $0$ and $1$ are elements of every field, one can use this matrix to construct subspaces of $\F_p^\mathcal P$, $p$ prime. There are two common ways to do this:
\begin{itemize}
  \item by considering the $\F_p$-row span of $A$, which is usually called \emph{the code generated by the $k$-spaces of $\mathcal S$} and denoted by $C_k(\mathcal S)$, and
  \item by considering the $\F_p$-null space of $A$, which is usually called \emph{the dual code of points and $k$-spaces of $\mathcal S$} and denoted by $C_k(\mathcal S)^\perp$.
\end{itemize}

\begin{remark}
  A \emph{code word} $c$ of $C_k(\mathcal S)^\perp$ is an element of the $\F_p$-null space of $A$, which is equivalent to a mapping from $\mathcal P$ to $\F_p$ with the additional property that $\sum_{p\in\pi}c_p=0$, for all $\pi\in\mathcal B$. Hence, code words can be studied as multisets of points such that each $k$-space on $\mathcal S$ contains $0\pmod p$ of the points in the multiset.
\end{remark}

These codes have primarily been studied for affine and projective spaces (see \cite{ak,fcvdt, conm} for an overview of relevant results), and more recently these codes have been investigated for quadrics, Hermitian varieties and generalized quadrangles \cite{dmm,psv,vdd1,vdd2}. In \cite{psv}, the dual code of points and $n$-spaces of the Hermitian variety $\mathcal{H}(2n+1,q^2)$ is studied. We will first introduce the basic definitions of Hermitian varieties, then we state the known results on the Hermitian variety code at hand, and finally we will state our main result.

From now on, let $\PG(m,q^{2})$ be the $m$-dimensional projective space over the finite field $\F_{q^{2}}$. A \emph{non-singular Hermitian variety} is the set of absolute points of a Hermitian polarity, which is defined by a Hermitian matrix and the non-trivial involution $x\mapsto x^{q}$ of $\Aut(\F_{q^{2}})$. All non-singular Hermitian varieties are projectively equivalent to the one given by the equation
\[
  X^{q+1}_{0}+X^{q+1}_{1}+\cdots+X^{q+1}_{m}=0\,.
\]
The projective index of a non-singular Hermitian variety in $\PG(m,q^{2})$ equals $\left\lfloor\frac{m-1}{2}\right\rfloor$. The maximal subspaces of a Hermitian variety are called \emph{generators}. From now on, we will denote the standard non-singular Hermitian variety in $\PG(m,q^{2})$ by $\mathcal{H}(m,q^{2})$. A \emph{singular Hermitian variety} in $\PG(m,q^{2})$ is a cone with an $i$-dimensional subspace as vertex and a non-singular Hermitian variety in an $(m-i-1)$-dimensional subspace, disjoint from the vertex, as base, $-1\leq i\leq m$. All singular Hermitian varieties in $\PG(m,q^{2})$ with an $i$-dimensional vertex are projectively equivalent to the one given by the equation
\[
  X^{q+1}_{0}+X^{q+1}_{1}+\cdots+X^{q+1}_{m-i-1}=0\,.
\]
Note that a non-singular Hermitian variety is a singular Hermitian variety with vertex dimension equal to $-1$.

\begin{lemma}\label{basiccounting}
  The number of generators on $\mathcal{H}(2n+1,q^{2})$ is $\prod^{n}_{i=0}(q^{2i+1}+1)$.
\end{lemma}
\begin{proof}
  This, and many other results on Hermitian varieties, can be found in \cite[Chapter 23]{ht}.
\end{proof}

\begin{notation}
  Throughout this article, we will denote the number of points in $\PG(m,q)$ by $\theta_{m}(q)=\frac{q^{m+1}-1}{q-1}$ and the number of points on $\mathcal{H}(m,q^{2})$ by $\mu_{m}(q^{2})=\frac{(q^{m+1}-(-1)^{m+1})(q^{m}-(-1)^{m})}{q^{2}-1}$.
\end{notation}

\begin{definition}
  The \emph{support} of a code word $c$ is the set of positions with nonzero entry, i.e. $\supp(c)=\{p:c_p\neq 0\}\subseteq\mathcal P$. Note that we identify the set of positions $\supp(c)$ with a set of points of $\mathcal{P}$ using the correspondence between those. The \emph{Hamming weight} $\wt(c)$ of a code word $c$ is the number of nonzero symbols in it, i.e. it is $|\supp(c)|$. The minimum distance $d$ of $C_k(\mathcal S)^\perp$ is $\min_{c\in C_k(\mathcal S)^\perp\setminus\{0\}} \wt(c)$.
\end{definition}

The following theorems on $C_n(\mathcal{H}(2n+1,q^2))^\perp$ are known.

\begin{theorem}[{\cite[Proposition 3.7]{kimms}}]\label{brol}
  Let $n=1$. Then the supports of all code words $c$ with $0<\wt(c)<3q$ are projectively equivalent, and their Hamming weights are $2(q+1)$. If $\wt(c)\leq \frac{\sqrt{q}(q+1)}{2}$, then $c$ is a linear combination of these code words.
\end{theorem}

\begin{theorem}[{\cite[Theorem 43]{psv}}]
  Let $n=2$. If $c$ is a code word with $\wt(c)\le 2(q^{3}+q^2)$ and if $q$ is sufficiently large, then there are only two possible projective equivalence classes for $\supp(c)$, and the Hamming weights of the corresponding code words are $2(q^3+1)$ and $2(q^3+q^2)$. These two types of code words are examples of the code words constructed in Theorem \ref{codewords}.
\end{theorem}

In this article, we will discuss the dual code arising from the points and generators of a Hermitian variety. This improves upon earlier work of \cite[Section 5]{psv}. We determine the minimum Hamming weight for general $n$, and we show that if $q$ is sufficiently large, a similar statement to the second part of Theorem \ref{brol} holds for general $n$. Our main result is as follows.

\begin{theorem}\label{maintheorem}
  Let $n$ be any positive integer and let $\delta>0$ be any constant. If $c$ is a code word with $\wt(c)<4q^{2n-1}$ and $q$ is sufficiently large, then there are only $n$ possible projective equivalence classes for $\supp(c)$
  ; call $S$ this set of projective equivalence classes. If $c$ is a code word with $\wt(c)<\delta q^{2n-1}$, then $c$ is a linear combination of code words which are an incidence vector of a set in $S$. The minimum distance of $C_{n}(\mathcal{H}(2n+1,q^{2}))^{\bot}$ is $2q^{2n-4}(q^3+1)$ for $n\ge 2$.
\end{theorem}

\section{The code words}\label{sec:codewords}

In this section we introduce a set of code words of the code $C_{n}(\mathcal{H}(2n+1,q^{2}))^{\bot}$. From now on, we consider the projective space $\PG(2n+1,q^{2})$, $n\geq1$.

\begin{lemma}\label{kegeltopalg}
  Consider a non-singular Hermitian variety $\mathcal{H}(2n+1,q^{2})$ in $\PG(2n+1,q^{2})$ and let $\sigma$ be the corresponding polarity. Let $\pi$ be a $k$-dimensional subspace in $\PG(2n+1,q^{2})$ such that $\pi\cap\mathcal{H}(2n+1,q^{2})$ is a cone $\pi_{i}H_{k-i-1}$ with $H_{k-i-1}\cong\mathcal{H}(k-i-1,q^{2})$ and $\pi_{i}$ an $i$-space, $-1\leq i\leq \min\{k,n\}$. Then $\pi\cap\pi^{\sigma}=\pi_{i}$. Conversely, if $\pi\cap\pi^{\sigma}$ is an $i$-space $\pi_{i}$, then $\pi\cap\mathcal{H}(2n+1,q^{2})$ is a cone $\pi_{i}H_{k-i-1}$ with $H_{k-i-1}\cong\mathcal{H}(k-i-1,q^{2})$.
\end{lemma}
\begin{proof}
 The first statement is \cite[Lemma 23.2.8]{ht}; the second statement is a corollary of the first.
\end{proof}

We will use this theorem mostly in the case $k=n$. 
Using the above lemma, we can prove an easy counting result.

\begin{theorem}\label{basiccounting2}
  The number of generators on $\mathcal{H}(2n+1,q^{2})$ through a fixed $k$-space on $\mathcal{H}(2n+1,q^{2})$, $0\leq k\leq n$, equals $\prod^{n-k-1}_{i=0}(q^{2i+1}+1)$.
\end{theorem}
\begin{proof}
  Let $\pi_{k}$ be a $k$-space on $\mathcal{H}(2n+1,q^{2})$ and let $\sigma$ be the polarity corresponding to $\mathcal{H}(2n+1,q^{2})$. Then $\pi^{\sigma}_{k}$ is a $(2n-k)$-space intersecting $\mathcal{H}(2n+1,q^{2})$ in a cone $\pi_{k}H$ with $H\cong\mathcal{H}(2n-2k-1,q^{2})$. Every generator on $\mathcal{H}(2n+1,q^{2})$ through $\pi_{k}$ corresponds uniquely to a generator on $H$. Hence, there are $\prod^{n-k-1}_{i=0}(q^{2i+1}+1)$ generators on $\mathcal{H}(2n+1,q^{2})$ through $\pi_{k}$.
\end{proof}

In the construction of the code words we need the following lemma.

\begin{lemma}\label{intersectgenerator}
  Let $\pi$ be an $n$-space in $\PG(2n+1,q^{2})$ and let $\mu$ be a generator of $\mathcal{H}(2n+1,q^{2})$. Then $\pi\cap\mu$ and $\pi^{\sigma}\cap\mu$ are subspaces of the same dimension.
\end{lemma}
\begin{proof}
  We denote $\mu\cap\pi=\pi_{j}$, a $j$-space, possibly empty ($j=-1$). It follows that $2n-j=\dim((\mu\cap\pi)^{\sigma})=\dim(\langle\mu^{\sigma},\pi^{\sigma}\rangle)$. Using the Grassmann identity and $\mu=\mu^{\sigma}$ ($\mu$ is a generator), we find $\dim(\mu\cap\pi^{\sigma})=\dim(\mu)+\dim(\pi^{\sigma})-\dim(\langle\mu^{\sigma},\pi^{\sigma}\rangle)=j$.
\end{proof}

Now, we can give the construction of small weight code words in the code $C_{n}(\mathcal{H}(2n+1,q^{2}))^{\bot}$. This construction is based on \cite[Theorem 58]{psv}.

\begin{theorem}\label{codewords}
  Consider $\mathcal{H}(2n+1,q^{2})$ and its corresponding polarity $\sigma$. Let $\pi$ be an $n$-space in $\PG(2n+1,q^{2})$. Denote the incidence vector of $\pi\cap\mathcal{H}(2n+1,q^{2})$ by $v_{\pi}$ and the incidence vector of $\pi^{\sigma}\cap\mathcal{H}(2n+1,q^{2})$ by $v_{\pi^{\sigma}}$. Then $\alpha(v_{\pi}-v_{\pi^{\sigma}})$, $\alpha\in\F_{p}$, is a code word of $C_{n}(\mathcal{H}(2n+1,q^{2}))^{\bot}$.
\end{theorem}
\begin{proof}
  Let $\mu$ be a generator of $\mathcal{H}(2n+1,q^{2})$ and denote its incidence vector by $v_{\mu}$. Using Lemma \ref{intersectgenerator}, we find $\mu$ intersects both $\pi$ and $\pi^{\sigma}$, or neither. In the first case $|\pi\cap\mu|\equiv|\pi^{\sigma}\cap\mu|\equiv 1\pmod{q}$ and in the second case $|\pi\cap\mu|=|\pi^{\sigma}\cap\mu|=0$. In both cases $v_{\pi}\cdot v_{\mu}=v_{\pi^{\sigma}}\cdot v_{\mu}$. The theorem follows.
\end{proof}

\begin{example}\label{possibilities}
  We list the different possibilities for $\pi\cap\pi^{\sigma}$. Hereby, we use Lemma \ref{kegeltopalg} for $k=n$ and Theorem \ref{codewords}. We write $\mathcal{H}=\mathcal{H}(2n+1,q^{2})$.
  \begin{itemize}
    \item $\pi\cap\pi^{\sigma}=\emptyset$. We write $\pi\cap\mathcal{H}=H$ and $\pi^{\sigma}\cap\mathcal{H}=H'$. We know, $H,H'\cong\mathcal{H}(n,q^{2})$. The corresponding code words have weight $2\mu_{n}(q^{2})$.
    \item $\pi\cap\pi^{\sigma}=\pi_{i}$, an $i$-space, $0\leq i\leq n-2$. We write $\pi\cap\mathcal{H}=\pi_{i}H$ and $\pi^{\sigma}\cap\mathcal{H}=\pi_{i}H'$, which are both cones, with $H,H'\cong\mathcal{H}(n-i-1,q^{2})$. The corresponding code words have weight $2q^{2i+2}\mu_{n-i-1}(q^{2})$. 
    \item $\pi\cap\pi^{\sigma}=\pi_{n-1}$, an $(n-1)$-space. Then $\pi\cap\mathcal{H}=\pi^{\sigma}\cap\mathcal{H}=\pi_{n-1}$ since $\mathcal{H}(0,q^{2})$ is empty. The construction gives rise to the zero code word.
    \item $\pi\cap\pi^{\sigma}=\pi_{n}$, an $n$-space. Then $\pi=\pi^{\sigma}=\pi_{n}\subset\mathcal{H}$. Also in this case, the construction gives rise to the zero code word.
  \end{itemize}
  It can easily be checked that among these four cases, the code words with smallest weight are the ones corresponding to $i=n-3$.
\end{example}

\begin{remark}\label{collinear}
  Consider the construction from Theorem \ref{codewords}, with $\pi\cap\mathcal{H}=\pi_{i}H_{n-i-1}$ and $\pi^{\sigma}\cap\mathcal{H}=\pi_{i}H'_{n-i-1}$. Let $P$ be a point of $\pi_{i}H_{n-i-1}$ and let $P'$ be a point of $\pi_{i}H'_{n-i-1}$. We know that $P'\in\pi^{\sigma}\subseteq P^{\sigma}$ and $P'\in\mathcal{H}$. Hence, the line $PP'$ is a line of $\mathcal{H}$. 
\end{remark}

\section{Some counting results}

\begin{lemma}\label{cnj}
  Consider the non-singular Hermitian variety $\mathcal{H}(2n+1,q^{2})\subset\PG(2n+1,q^{2})$ and let $\sigma$ be the corresponding polarity. Let $\tau$ be a $j$-space such that $\tau\cap\mathcal{H}(2n+1,q^{2})=H_{j}\cong\mathcal{H}(j,q^{2})$, $-1\leq j\leq n$. The number of generators on $\mathcal{H}(2n+1,q^{2})$ skew to $\tau$ equals
  \[
    c_{n,j}:=q^{\binom{j+1}{2}}\prod^{n-j-1}_{k=0}(q^{2k+1}+1)\prod^{2n-j+1}_{l=2(n-j)+1}(q^{l}-(-1)^{l}).
  \]
\end{lemma}
\begin{proof}
  By \cite[Theorem 23.4.2 (i)]{ht} we know that the number of generators skew to $\tau$ only depends on the parameters $n$ and $j$ and not on the choice of $\tau$ itself.
  \par We will prove this theorem using induction on $j$. If $j=-1$, $\tau$ is the empty space and hence $c_{n,-1}$ equals the total number of generators. By Lemma \ref{basiccounting} we find $c_{n,-1}=\prod^{n}_{k=0}(q^{2k+1}+1)$. Now, we prove that a relation between $c_{n,j}$ and $c_{n-1,j-1}$ holds.
  \par By Lemma \ref{kegeltopalg} we know $\tau\cap\tau^{\sigma}=\emptyset$. Hence, every point $P\in\PG(2n+1,q^{2})\setminus(\tau\cup\tau^{\sigma})$ can uniquely be written as $P_{\tau}+\lambda_{P}P_{\tau^{\sigma}}$, $P_{\tau}\in\tau$, $P_{\tau^{\sigma}}\in\tau^{\sigma}$, $\lambda_{P}\in\F^{*}_{q^{2}}$. For every point $P\in\PG(2n+1,q^{2})\setminus(\tau\cup\tau^{\sigma})$, we define $\phi_{\tau}(P)=P_{\tau}$. This is the projection of $P$ from $\tau^{\sigma}$ on $\tau$. We define a correlation $\overline{\sigma}:\tau\to\tau$ that maps the subspace $U\subset\tau$ to $U^{\sigma}\cap\tau$. It is straightforward to check that $\overline{\sigma}$ defines a polarity on $\tau$. Moreover, it can be seen easily that the points of $H_{j}$ are the absolute points of $\overline{\sigma}$. Hence, $\overline{\sigma}$ is the polarity of $\tau$ corresponding to $H_{j}$.
  \par Now, we consider the set $S=\{(P,\mu)\mid P\in\mu\setminus\tau^{\sigma},\phi_{\tau}(P)\notin H_{j},\mu\text{ a generator}, \mu\cap\tau=\emptyset\}$. We count the number of elements of $S$ in two ways. On the one hand, there are $c_{n,j}$ generators skew to $\tau$. Let $\mu$ be such a generator. The intersection $\mu\cap\tau^{\sigma}$ is an $(n-j-1)$-space since $\dim(\mu\cap\tau^{\sigma})+\dim(\langle\mu,\tau\rangle)=2n$. We also know that $\phi_{\tau}(P)=R$ for every point $P\in\langle R,\mu\cap\tau^{\sigma}\rangle\setminus(\mu\cap\tau^{\sigma})$, $R\in\tau$. Hence, for each generator there are $\theta_{n}(q^{2})-\theta_{n-j-1}(q^{2})-\mu_{j}(q^{2})(\theta_{n-j}(q^{2})-\theta_{n-j-1}(q^{2}))=q^{2(n-j)}(\theta_{j}(q^{2})-\mu_{j}(q^{2}))$ points fulfilling the requirements. On the other hand, we count the points $P\in\mathcal{H}(2n+1,q^{2})\setminus(\tau\cup\tau^{\sigma})$ fulfilling the requirements. There are $\mu_{2n+1}(q^{2})-\mu_{j}(q^{2})-\mu_{2n-j}(q^{2})$ points in this set. We must assure that $\phi_{\tau}(P)\notin H_{j}$. Let $R$ be a point of $H_{j}$. Since $\tau^{\sigma}\subseteq R^{\sigma}$, a line $RQ$, $Q\in\tau^{\sigma}$, is a tangent line (in $R$) to $\mathcal{H}(2n+1,q^{2})$ or a line which is completely contained in $\mathcal{H}(2n+1,q^{2})$. Hence, $\phi_{\tau}(P)$ is a point of $H_{j}$ iff $P$ lies on a line through $\phi_{\tau}(P)$ and a point of $\tau^{\sigma}\cap\mathcal{H}(2n+1,q^{2})$. Consequently there are
  \begin{align*}
    &\mu_{2n+1}(q^{2})-\mu_{j}(q^{2})-\mu_{2n-j}(q^{2})-\mu_{j}(q^{2})\mu_{2n-j}(q^{2})(q^{2}-1)\\=\ &q^{2n-j}(\theta_{j}(q^{2})-\mu_{j}(q^{2}))(q^{2n-j+1}-(-1)^{2n-j+1})
  \end{align*}
  points $P\in\mathcal{H}(2n+1,q^{2})\setminus(\tau\cup\tau^{\sigma})$ fulfilling the requirement $\phi_{\tau}(P)\notin H_{j}$. Now, we fix such a point $P$ and we count the number of generators skew to $\tau$, through it. All these generators are contained in $P^{\sigma}$. We know $P^{\sigma}\cap\mathcal{H}(2n+1,q^{2})$ is a cone $PH_{2n-1}$, with $H_{2n-1}\cong\mathcal{H}(2n-1,q^{2})$. There is a 1-1 correspondence between the generators of $\mathcal{H}(2n+1,q^{2})$ through $P$ and the generators of $H_{2n-1}$. We also find
  \[
    \tau\cap P^{\sigma}=\tau\cap(\phi_{\tau}(P)+\lambda P')^{\sigma}=\tau\cap((\phi_{\tau}(P))^{\sigma}+\lambda^{q} P'^{\sigma})=\tau\cap(\phi_{\tau}(P))^{\sigma}=(\phi_{\tau}(P))^{\overline{\sigma}},
  \]
  with $P'\in\tau^{\sigma}$ (and thus $\tau\subset P'^{\sigma}$). Hence, the $(j-1)$-space $\tau\cap P^{\sigma}$ intersects $\mathcal{H}(2n+1,q^{2})$ in $H_{j-1}\cong\mathcal{H}(j-1,q^{2})$, since $\phi_{\tau}(P)\notin H_{j}$. We can choose the base of the cone $PH_{2n-1}$ such that it contains $\tau\cap P^{\sigma}$. The generators through $P$ and skew to $\tau$ correspond to the generators of $H_{2n-1}$, skew to $\tau\cap P^{\sigma}$. There are $c_{n-1,j-1}$ such generators. We conclude
  \begin{align*}
    c_{n,j}q^{2(n-j)}\left[\theta_{j}(q^{2})-\mu_{j}(q^{2})\right]&=c_{n-1,j-1}q^{2n-j}\left[\theta_{j}(q^{2})-\mu_{j}(q^{2})\right](q^{2n-j+1}-(-1)^{2n-j+1})\\
    \Rightarrow \qquad c_{n,j}&=c_{n-1,j-1}q^{j}(q^{2n-j+1}-(-1)^{2n-j+1}).
  \end{align*}
  \par An induction calculation now finishes the proof.
\end{proof}

From now on in this section, we use the following notation: $H\cong\mathcal{H}(2n+1,q^{2})$  is a non-singular Hermitian variety and $\sigma$ is the polarity corresponding to it; $\pi$ is an $n$-space in $\PG(2n+1,q^{2})$, such that $\pi\cap H$ is a cone $\pi_{i}H_{n-i-1}$ with $H_{n-i-1}\cong\mathcal{H}(n-i-1,q^{2})$ and $\pi_{i}$ an $i$-space, $-1\leq i\leq n$. By Lemma \ref{kegeltopalg}, for $k=n$, we know $\pi\cap\pi^{\sigma}=\pi_{i}$ and consequently $\pi^{\sigma}\cap H$ is a cone $\pi_{i}H'_{n-i-1}$ with $H'_{n-i-1}\cong\mathcal{H}(n-i-1,q^{2})$.

\begin{definition}
  The number of generators on $H$ intersecting $\pi$ in a fixed point $P\in\pi_{i}H_{n-i-1}\setminus\pi_{i}$ and no other point of $\pi_{i}H_{n-i-1}$, and intersecting $\pi^{\sigma}$ in a fixed point $P'\in\pi_{i}H'_{n-i-1}\setminus\pi_{i}$ and no other point of $\pi_{i}H'_{n-i-1}$ is denoted by $N(\pi,P,P')$. The number of generators on $H$ skew to $\pi$ is denoted by $N'(\pi)$.
\end{definition}

By Lemma \ref{intersectgenerator} we know that the generators skew to $\pi$ are also skew to $\pi^{\sigma}$ and that the generators intersecting $\pi$ in precisely one point also intersect $\pi^{\sigma}$ in precisely one point.

\begin{lemma}
  The number $N'(\pi)$ only depends on the intersection parameters $(n,i)$ of $\pi$.
\end{lemma}
\begin{proof}
  This follows immediately from \cite[Theorem 23.4.2 (i)]{ht}. 
\end{proof}

\begin{notation}
  Consequently, we can denote $N'(\pi)$ by $N'(n,i)$.
\end{notation}

\begin{lemma}\label{NlinkedN'}
  For $n\geq 2$, $-1\leq i\leq n-2$, $N(\pi,P,P')=N'(n-2,i)$. Consequently, $N(\pi,P,P')$ only depends on the intersection parameters $(n,i)$ of $\pi$.
\end{lemma}
\begin{proof}
  Consider the points $P\in(\pi_{i}H_{n-i-1}\setminus\pi_{i})\subseteq\pi$ and $P'\in(\pi_{i}H'_{n-i-1}\setminus\pi_{i})\subseteq\pi^{\sigma}$. Denote $\ell=\langle P,P'\rangle$. Then $\ell^{\sigma}$ is a $(2n-1)$-space intersecting $H$ in a cone with $\ell$ as vertex and a non-singular $(2n-3)$-dimensional Hermitian variety $H_{2n-3}$ as base. 
  Since $\dim(\ell\cap\pi)=\dim(\ell\cap\pi^{\sigma})=0$, $\ell^{\sigma}\cap\pi=V$ is an $(n-1)$-space and $\ell^{\sigma}\cap\pi^{\sigma}=V'$ is an $(n-1)$-space. Also, $\ell\subset\langle\pi,\pi^{\sigma}\rangle=\pi^{\sigma}_{i}$, hence $\pi_{i}\subset l^{\sigma}$. Let $W$, resp. $W'$, be an $(n-2)$-space in $V$, resp. $V'$, containing $\pi_{i}$ and not through $P$, resp. $P'$. Denote the $(2n-i-4)$-space $\left\langle W,W'\right\rangle$ by $\tau'$. It can be seen that on the one hand $\tau'\subset\ell^{\sigma}$ and on the other hand $\ell\cap\tau'=\emptyset$, so the $(2n-3)$-space $\tau$ containing the base $H_{2n-3}$ can be chosen such that $\tau'\subseteq\tau$. Let $\sigma'$ be the polarity of $\tau$ corresponding to $H_{2n-3}$. Analogously to the proof of Lemma \ref{cnj} we can define this polarity as follows: $U^{\sigma'}=\tau\cap U^{\sigma}$. It now immediately follows that $W^{\sigma'}=W'$ because both are $(n-2)$-spaces contained in $W^{\sigma}$ and in $\tau$.
  \par Arguing as in the proof of Lemma \ref{cnj}, we see there is a 1-1 correspondence between the generators of $H_{2n-3}$ and the generators of $H$ through $\ell$ (the generators containing $P$ and $P'$). If a generator of $H$ through $\ell$ contains no points of $\pi\cup\pi^{\sigma}$ but $P$ and $P'$, then its corresponding generator of $H_{2n-3}$ is skew to $W$ and $W'$. Vice versa, every generator $\mu$ of $H_{2n-3}$ skew to $W$ and $W'$, is contained in precisely one generator of $H$ intersecting  $\pi\cup\pi^{\sigma}$ in only the points $P$ and $P'$, namely $\langle\mu,P,P'\rangle$. Since $W^{\sigma'}=W'$, the generators of $H_{2n-3}$ skew to $W$ and $W'$ are the ones skew to $W$, by Lemma \ref{intersectgenerator}. Hence, $N(\pi,P,P')=N'(n-2,i)$.
  \par The second statement of the lemma follows immediately from the first one.
\end{proof}

\begin{notation}
  Since $N(\pi,P,P')$ only depends on the intersection parameters $(n,i)$ of $\pi$, we can denote it by $N(n,i)$.
\end{notation}

The previous theorem now states $N(n,i)=N'(n-2,i)$ for $n\geq2$, $-1\leq i\leq n-2$.

\begin{lemma}\label{formuleN}
  For $n\geq 1$ and $-1\leq i\leq n-2$, the following equality holds:
  \[
    N(n,i)=q^{(n-1)^{2}-\binom{n-i-1}{2}}\prod^{n-i-2}_{j=1}(q^{j}-(-1)^{j}).
  \]
\end{lemma}
\begin{proof}
  We prove this theorem using induction. Using Lemma \ref{NlinkedN'}, we know that $N(n,-1)$ equals $N'(n-2,-1)$, the number of generators of a Hermitian variety $H'\cong\mathcal{H}(2n-3,q^{2})$ skew to an $(n-2)$-space intersecting $H'$ in a Hermitian variety $\mathcal{H}(n-2,q^{2})$, if $n\geq2$. This number equals $c_{n-2,n-2}$. Hence, by Lemma \ref{cnj},
  \[
    N(n,-1)=q^{\binom{n-1}{2}}\prod^{n-1}_{l=1}(q^{l}-(-1)^{l})=q^{(n-1)^{2}-\binom{n-(-1)-1}{2}}\prod^{n-(-1)-2}_{j=1}(q^{j}-(-1)^{j}),
  \]
  which proves the induction base for $n\geq2$. If $n=1$, it is easy to prove that $N(1,-1)=1$. Hence, the formula holds also in this case.
  \par Now, we will prove that $N(n,i)=q^{2n-3}N(n-1,i-1)$. By Lemma \ref{NlinkedN'}, this is equivalent to proving that $N'(n,i)=q^{2n+1}N'(n-1,i-1)$. Consider the set $S=\{(R,\mu)\mid R\in\mu,\mu \text{ a generator skew to }\pi, R\notin\langle\pi,\pi^{\sigma}\rangle=\pi^{\sigma}_{i}\}$. This subspace $\pi^{\sigma}_{i}$ intersects $H$ in a cone $\pi_{i}H_{2(n-i)-1}$. We will count $|S|$ in two ways.
  \par On the one hand, there are $N'(n,i)$ generators skew to $\pi$. Fix such a generator $\mu$. Then $\dim(\mu\cap\pi^{\sigma}_{i})=n-i-1$ since $\dim(\mu\cap\pi^{\sigma}_{i})+\dim(\langle\mu,\pi_{i}\rangle)=2n$. So, $\mu$ contains precisely $\theta_{n}(q^{2})-\theta_{n-i-1}(q^{2})=q^{2(n-i)}\theta_{i}(q^{2})$ points of $\PG(2n+1,q^{2})\setminus\pi^{\sigma}_{i}$. Consequently, $|S|=q^{2(n-i)}\theta_{i}(q^{2})N'(n,i)$.
  \par On the other hand, there are $\mu_{2n+1}(q^{2})-\theta_{i}(q^{2})-\mu_{2(n-i)-1}(q^{2})-(q^{2}-1)\theta_{i}(q^{2})\mu_{2(n-i)-1}(q^{2})=q^{4n-2i+1}\theta_{i}(q^{2})$ points in $H\setminus\pi^{\sigma}_{i}$. Fix such a point $P$. The hyperplane $P^{\sigma}$ intersects $\pi$ in an $(n-1)$-space $V$ and intersects $\pi_{i}$ in an $(i-1)$-space $\pi_{i-1}\subset V$. Hence, the intersection $V\cap H$ has intersection parameters $(n-1,i-1)$. The intersection $P^{\sigma}\cap H$ is a cone $PH_{2n-1}$, with $H_{2n-1}\cong\mathcal{H}(2n-1,q^{2})$. Let $\tau$ be the $(2n-1)$-space containing $H_{2n-1}$. We can choose $\tau$ such that it contains $V$. Then, there is a 1-1 correspondence between the generators of $\mathcal{H}(2n+1,q^{2})$ through $P$, skew to $\pi$ and the generators of $H_{2n-1}$ skew to $V$. Consequently, there are $N'(n-1,i-1)$ such generators. Thus, $|S|=q^{4n-2i+1}\theta_{i}(q^{2})N'(n-1,i-1)$.
  \par Comparing both expressions for $|S|$,  we find the desired relation between $N'(n,i)$ and $N'(n-1,i-1)$. An easy calculation now finishes the proof.
\end{proof}

\begin{lemma}\label{generatorsthroughpoint}
  Assume $n\geq2$ and $-1\leq i\leq n-2$. Let $P$ be a point of $H\setminus(\pi\cup\pi^{\sigma})$. Let $n_{P}(n,i)$ be the number of generators through $P$ intersecting both $\pi\setminus\pi_{i}$ and $\pi^{\sigma}\setminus\pi_{i}$ in precisely one point. Then,
  \begin{itemize}
    \item $n_{P}(n,i)=N(n-1,i-1)q^{4i}\left(\mu_{n-i-1}(q^{2})\right)^{2}$ if $P\notin\langle\pi,\pi^{\sigma}\rangle=\pi^{\sigma}_{i}$;
    \item $n_{P}(n,i)=N(n-1,i)q^{4i+4}\left(\mu_{n-i-2}(q^{2})\right)^{2}$ if $P\in\langle\pi,\pi^{\sigma}\rangle=\pi^{\sigma}_{i}$ but $P$ does not lie on a line of $H$ through a point of $\pi\setminus\pi_{i}$ and a point of $\pi^{\sigma}\setminus\pi_{i}$;
    \item $n_{P}(n,i)=N(n-1,i+1)q^{4i+4}\mu_{n-i-3}(q^{2})\left[q^{4}\mu_{n-i-3}(q^{2})+q^{2}-1\right]$ if $P\in\langle\pi,\pi^{\sigma}\rangle=\pi^{\sigma}_{i}$ and $P$ lies on a line of $H$ through a point of $\pi\setminus\pi_{i}$ and a point of $\pi^{\sigma}\setminus\pi_{i}$, and $i\leq n-4$.
    \item $n_{P}(n,i)=q^{2i+2}N(n,i)$ if $P\in\langle\pi,\pi^{\sigma}\rangle=\pi^{\sigma}_{i}$ and $P$ lies on a line of $H$ through a point of $\pi\setminus\pi_{i}$ and a point of $\pi^{\sigma}\setminus\pi_{i}$, and $i=n-3,n-2$.
  \end{itemize}
  The first case can only occur if $i\geq0$. The second case can only occur if $i\leq n-3$. 
\end{lemma}
\begin{proof}
  Since $P\notin\pi\cup\pi^{\sigma}$, $P^{\sigma}\cap\pi=V$ is an $(n-1)$-space and $P^{\sigma}\cap\pi^{\sigma}=V'$ is an $(n-1)$-space. Furthermore $P^{\sigma}\cap H$ is a cone $PH_{2n-1}$ with $H_{2n-1}\cong\mathcal{H}(2n-1,q^{2})$. Let $\tau$ be the $(2n-1)$-space containing $H_{2n-1}$.
  \par First we consider the case $P\notin\langle\pi,\pi^{\sigma}\rangle=\pi^{\sigma}_{i}$. In this case $P^{\sigma}$ intersects $\pi_{i}$ in an $(i-1)$-space $\pi_{i-1}=V\cap V'$. Also, $\tau$ can be chosen so that it contains $V$ and $V'$. Hence, the number of generators through $P$ fulfilling the requirements equals the number of generators of $H_{2n-1}$ intersecting $V$ and $V'$ in a point. Let $\sigma'$ be the polarity of $\tau$ corresponding to $H_{2n-1}$. Analogously to the argument in the proof of Lemma \ref{NlinkedN'}, it can be seen that $V'=V^{\sigma'}$. Consequently there are $N(n-1,i-1)$ generators of this type through a fixed point of $V\setminus\pi_{i-1}$ and a fixed point of $V'\setminus\pi_{i-1}$. There are $q^{2i}\mu_{n-i-1}(q^{2})$ possible choices for each of these points. The first part of the lemma follows. Note that $\langle\pi,\pi^{\sigma}\rangle=\PG(2n+1,q^{2})$ if $i=-1$. Hence, this case cannot occur if $i=-1$.
  \par We fix some notation for the remaining cases. Let $W\subseteq\pi$ and $W'\subseteq\pi^{\sigma}$ be the $(n-i-1)$-spaces containing $H_{n-i-1}$ and $H'_{n-i-1}$, respectively. Furthermore, let $\overline{\sigma}$ and $\overline{\sigma}'$ be the polarities of $W$ and $W'$, respectively corresponding to $H_{n-i-1}$ and $H'_{n-i-1}$. In all three remaining cases, $\pi_{i}\subset P^{\sigma}$, hence $P^{\sigma}\cap W=W_{1}$ and $P^{\sigma}\cap W'=W'_{1}$ are $(n-i-2)$-spaces. Now, the point $P$ can be written in a unique way as $P=\lambda P_{\pi_{i}}+\lambda_{P}P_{W}+P_{W'}$, with $P_{W}\in W$, $P_{W'}\in W'$, $P_{\pi_{i}}\in\pi_{i}$ and $\lambda,\lambda_{P}\in\F_{q^{2}}$. Arguing as in the proof of Lemma \ref{cnj} we can see that $W_{1}=P^{\sigma}\cap W=P^{\overline{\sigma}}_{W}$ and that $W'_{1}=P^{\sigma}\cap W'=P^{\overline{\sigma}'}_{W'}$. Moreover, since $P$ and $P_{\pi_{i}}$ are contained in $P^{\sigma}$, neither or both of $P_{W}$ and $P_{W'}$ are contained in $P^{\sigma}$. Hence, we need to distinguish two cases.
  \begin{itemize}
    \item $P_{W}\in W_{1}$ and $P_{W'}\in W'_{1}$ are both contained in $P^{\sigma}$; consequently, $P_{W}\in P^{\overline{\sigma}}_{W}$, thus $P_{W}\in H_{n-i-1}\subset H$ and $P^{\overline{\sigma}}_{W}\cap H_{n-i-1}$ is a cone $P_{W}H_{n-i-3}$, with $H_{n-i-3}\cong\mathcal{H}(n-i-3,q^{2})$. Let $W_{2}\subset W_{1}$ be the $(n-i-3)$-space containing $H_{n-i-3}$. Then, the intersection of $V=\langle\pi_{i},W_{1}\rangle$ and $H$ is the cone with vertex $\langle\pi_{i},P_{W}\rangle$ and base $H_{n-i-3}$. Analogously we introduce $H'_{n-i-3}\subset W'_{2}\subset W'_{1}$. Then $V'\cap H$ is the cone with vertex $\langle\pi_{i},P_{W'}\rangle$ and base $H'_{n-i-3}$. Furthermore, since $P_{W}\in V$, $P_{W'}\in V'$, and $P_{\pi_{i}}\in V,V'$, $P$ is contained in $\langle V,V'\rangle$. Also, the line $PP_{W}$ is contained in $P^{\sigma}$ and is not a $1$-secant since $P,P_{W}\in H$, hence it is a line of $H$. This line intersects $\pi^{\sigma}$ in a point of $\langle P_{W'},\pi_{i}\rangle\setminus\pi_{i}$.
    \item $P_{W}\notin W_{1}$ and $P_{W'}\notin W'_{1}$ are both not contained in $P^{\sigma}$; consequently, $P_{W}\notin P^{\overline{\sigma}}_{W}$, thus $P_{W}\notin H_{n-i-1}$, $P_{W}\notin H$ and $P^{\overline{\sigma}}_{W}\cap H_{n-i-1}$ is a non-singular Hermitian variety $H_{n-i-2}\cong\mathcal{H}(n-i-2,q^{2})$ in $W_{1}$. Then, the intersection of $V=\langle\pi_{i},W_{1}\rangle$ and $H$ is the cone $\pi_{i}H_{n-2-i}$. Analogously we introduce $H'_{n-i-2}\subset W'_{1}$. The intersection $V'\cap H$ is the cone $\pi_{i}H'_{n-i-2}$. Furthermore, $P\notin\langle V,V'\rangle$ since $P_{W}\notin W_{1}$ and $P_{W'}\notin W'_{1}$. Also, all lines in $\pi^{\sigma+}_{i}$ through $P$ intersecting $\pi\setminus\pi_{i}$ and $\pi^{\sigma}\setminus\pi_{i}$, are contained in $\langle P_{W},P_{W'},\pi_{i}\rangle$, but not in $\langle P,\pi_{i}\rangle$. Since $P_{W}, P_{W'}\notin P^{\sigma}$, none of the lines through $P$ can be contained in $H$.
  \end{itemize}
  These two cases clearly correspond to the three remaining cases of the lemma. We will treat them separately.
  \par First of all, we look at the latter, which is the second case in the statement of the lemma. Since $P\notin\langle V,V'\rangle$, we can choose $\tau$ such that it contains $\langle V,V'\rangle$. Hence, every generator through $P$, intersecting both $\pi\setminus\pi_{i}$ and $\pi^{\sigma}\setminus\pi_{i}$ in a point, corresponds to a generator of $H_{2n-1}$ intersecting both $V\setminus\pi_{i}$ and $V'\setminus\pi_{i}$ in a point, and vice versa. For a fixed point in $V\setminus\pi_{i}$ and a fixed point in $V'\setminus\pi_{i}$, there are $N(n-1,i)$ such generators. We also know that $|V\setminus\pi_{i}|=|V'\setminus\pi_{i}|=q^{2i+2}\mu_{n-i-2}(q^{2})$. The second part of the lemma follows. Note that $V\setminus\pi_{i}$ and $V'\setminus\pi_{i}$ are empty if $i=n-2$. Hence, this case only occurs if $i\leq n-3$.
  \par Finally, we look at the former case, the third and the fourth case in the statement of the lemma. Let $\ell$ be a line on $H$ through $P$, a point of $\pi\setminus\pi_{i}$ and a point of $\pi^{\sigma}\setminus\pi_{i}$. By changing, if necessary, the choices for $W$ and $W'$, we can assume $\ell=P_{W}P_{W'}$. We distinguish between two types of generators: the ones that contain $\ell$ and the ones that do not contain $\ell$. First we look at the ones that contain $\ell$. We know $\ell^{\sigma}\cap H$ is a cone with vertex $\ell$ and base $H_{2n-3}\cong\mathcal{H}(2n-3,q^{2})$. Let $\tau'$ be the $(2n-3)$-space containing $H_{2n-3}$. We can choose $\tau'$ so that it contains $\pi_{i}$, $W_{2}$ and $W'_{2}$. As before, one can see that $\langle\pi_{i},W_{2}\rangle^{\widehat{\sigma}'}=\langle\pi_{i},W'_{2}\rangle$, with $\widehat{\sigma}'$ the polarity of $\tau'$ corresponding to $H_{2n-3}$. The number of generators of the requested type through $\ell$ then equals the number of generators of $H_{2n-3}$ skew to $\langle\pi_{i},W_{2}\rangle$. This number equals $N'(n-2,i)=N(n,i)$. Furthermore, since $\ell$ is a line on $H$ through $P$ intersecting $\pi\setminus\pi_{i}$ and $\pi^{\sigma}\setminus\pi_{i}$, every line through $P$ and a point of $\langle P_{W},\pi_{i}\rangle\setminus\pi_{i}$ lies on $H$ and intersects $\langle P_{W'},\pi_{i}\rangle\setminus\pi_{i}\subset\pi^{\sigma}\setminus\pi_{i}$. Thus, there are $\theta_{i+1}(q^{2})-\theta_{i}(q^{2})=q^{2i+2}$ such lines. Hence, there are $q^{2i+2}N(n,i)$ generators of the first type. Now, we assume no line through $P$, intersecting $\pi$ and $\pi^{\sigma}$, is contained in the generator. Let $Q_{W}$ and $Q_{W'}$ be the points of the generator in $W$ and $W'$, respectively. By the previous remarks on this case, we know there are $\mu_{n-i-3}(q^{2})q^{2i+4}$ possible choices for $Q_{W}$ and for $Q_{W'}$. Now, we consider the plane $\langle P,Q_{W},Q_{W'}\rangle$. Using arguments, similar to the ones in the previous case, we find $N'(n-3,i+1)=N(n-1,i+1)$ generators fulfilling the requirements for every choice of $Q_{W}$ and $Q_{W'}$. Hence, the total number of generators in this third case equals
  \begin{align*}
    n_{P}&=q^{2i+2}N(n,i)+\left(\mu_{n-i-3}(q^{2})q^{2i+4}\right)^{2}N(n-1,i+1)\\
    &=\left[q^{2i+2}q^{2i+2}(q^{2}-1)\mu_{n-i-3}(q^{2})+\left(\mu_{n-i-3}(q^{2})q^{2i+4}\right)^{2}\right]N(n-1,i+1)\\
    &=q^{4i+4}\mu_{n-i-3}(q^{2})\left[q^{2}-1+q^{4}\mu_{n-i-3}(q^{2})\right]N(n-1,i+1)\;.\\
  \end{align*}
  Hereby we used the relation between $N(n,i)$ and $N(n-1,i+1)$ which can immediately be derived from Lemma \ref{formuleN}.
  \par Note that $V\setminus\langle\pi_{i},P_{W}\rangle$ and $V'\setminus\langle\pi_{i},P_{W'}\rangle$ are empty if $n-3\leq i\leq n-2$. In this case, we cannot consider the points $Q_{W}$ and $Q_{W'}$. So, there are no generators of the second type. Consequently, all generators are of the first type and there are precisely $q^{2i+2}N(n,i)$ such generators.
\end{proof}

\section{Classifying the small weight code words}

Before giving the new classification theorem, we state the results about the codes $C_{1}(\mathcal{H}(3,q^{2}))^{\bot}$ and $C_{2}(\mathcal{H}(5,q^{2}))^{\bot}$ to which we referred earlier.

\begin{theorem}[{\cite[Proposition 3.7]{kimms}}]
  Let $C$ be the code $C_{1}(\mathcal{H}(3,q^{2}))^{\bot}$. There is only one non-trivial type of code words among the ones described in Example \ref{possibilities}, namely $i=-1$. These are the code words of minimal weight. Let $c$ be a code word of $C$ with $\wt(c)\leq \frac{\sqrt{q}(q+1)}{2}$. Then $c$ is a linear combination of code words of minimal weight.
\end{theorem}

\begin{theorem}[{\cite[Theorem 43]{psv}}]
  Let $c$ be a code word of $C_{2}(\mathcal{H}(5,q^{2}))^{\bot}$, $q>893$, with $\wt(c)\leq 2(q^{3}+q)$, then $c$ is a code word of one of the types described in Theorem \ref{codewords}. Regarding Example \ref{possibilities}, we know that there are precisely two possibilities since $n=2$, namely $i=-1$ and $i=0$.
\end{theorem}

It is our aim to generalise this result. We start our arguments with two lemmas about $n$-spaces: the second lemma shows the existence of an $n$-space containing a non-trivial amount of points of the support of a code word, while the first lemma shows that a generator cannot contain many points of the support of a code word. In the proof of the second lemma we use the following result.

\begin{theorem}\label{pperpcapS}
  Let $c\in C_{n}(\mathcal{H}(2n+1,q^{2}))^{\bot}$ be a code word and denote $\supp(c)=S$. Let $P$ be a point in $S$. Then $|P^{\sigma}\cap S|\geq 2+q^{2n-1}$.
\end{theorem}
\begin{proof}
  This is a special case of \cite[Proposition 9(d)]{psv}.
\end{proof}

Throughout the three following lemmas the value $\Sigma_{n,i}$ is used, $-1\leq i\leq n-2$. It is defined by
\[
  \Sigma_{n,i}=\begin{cases}
                 2q^{2i+2}\mu_{n-i-1}(q^{2})+4\frac{\mu_{n-i-2}(q^{2})(q^{n-i-1}-1)}{q^{n-3i-5}(q^{2}-1)}&n-i\text{ odd}\;,\\
                 2q^{2i+2}\left[\mu_{n-i-1}(q^{2})+2\frac{q^{4}\mu_{n-i-3}(q^{2})+q^{2}-1}{q^{2}-1}\right]&n-i\text{ even}\;.
               \end{cases}\;.
\]
Note that in both cases $\Sigma_{n,i}=2q^{2n-1}+f$, with $f\in \mathcal{O}(q^{2n-2})$ and $f>0$ if $q>0$.

\begin{lemma}\label{isotrope}
  Let $c\in C_{n}(\mathcal{H}(2n+1,q^{2}))^{\bot}$ be a code word with $\wt(c)\leq w=\delta q^{2n-1}$, and denote $\supp(c)=S$. Let $\pi$ be a generator of $\mathcal{H}(2n+1,q^{2})$. Then $|\pi\cap S|\leq\delta\theta_{n-1}(q^{2})$.
\end{lemma}
\begin{proof}
  The proof is a generalisation of the proof of \cite[Lemma 41]{psv}.
  \par Denote $x=|\pi\cap S|$ and let $P$ be a point in $\pi\cap S$. Then $P^{\sigma}\cap\mathcal{H}(2n+1,q^{2})$ is a cone with vertex $P$. Let $H'\cong\mathcal{H}(2n-1,q^{2})$ be a base of this cone and consider the projection from $P$ onto $H'$. Denote the projection of $S\cap P^{\sigma}$ by $S'$. The projection of $\pi$ is a generator $\pi'$ of $H'$. Note that $S'$ is a blocking set of the generators on $H'$.
  \par By \cite[Lemma 10]{kms}, we know there are $q^{n^{2}}$ generators in $H'$ that are skew to $\pi'$, of which $q^{(n-1)^{2}}$ pass through a fixed point of $H'\setminus\pi'$. Hence, the blocking set $S'$ contains at least $q^{2n-1}$ points not in $\pi'$. Counting the tuples $(P,Q)$, $P\in\pi\cap S$, $Q\in S\setminus\pi$, with $PQ\subset\mathcal{H}(2n+1,q^{2})$, in two ways we find
  \[
    xq^{2n-1}\leq\delta q^{2n-1}\theta_{n-1}(q^{2})\;,
  \]
  where the upper bound follows from the fact that every point $Q\in S\setminus\pi$ is collinear with the points of an $(n-1)$-space in $\pi$ and not with the other points in $\pi$. The theorem follows immediately.
\end{proof}

Note that the size of a blocking set on a Hermitian variety $\mathcal{H}(2n+1,q^{2})$ is at least the size of an ovoid, hence at least $q^{2n+1}+1$.
\par Recall that the symmetric difference $A\Delta B$ of two sets $A$ and $B$ is the set $(A\cup B)\setminus(A\cap B)$. 

\begin{lemma}\label{subspace}
  Let $p$ be a fixed prime and denote $q=p^h$, $h\in\N$. Let $c\in C_{n}(\mathcal{H}(2n+1,q^{2}))^{\bot}$ be a code word with $\wt(c)\leq w=\delta q^{2n-1}$, $\delta>0$ a constant, and denote $\supp(c)=S$. Denote $\mathcal{H}(2n+1,q^{2})$ by $H$ and let $\sigma$ be the polarity related to $H$. Then a constant $C_{n}>0$, a value $Q>0$ and an $n$-space $\pi$ can be found such that $|(\pi\Delta\pi^{\sigma})\cap S|> C_{n} q^{2n-1}$ and such that $\frac{p-1}{p}|(\pi\Delta\pi^{\sigma})\cap H|<\Sigma_{n,i}-C_{n}q^{2n-1}$, if $q\geq Q$. Hereby, $i$ is such that $\pi\cap H$ is a cone with an $i$-dimensional vertex and $i\leq n-2$.
\end{lemma}
\begin{proof}
  We introduce the notion of a {\it semi-arc}. A semi-arc $\mathcal{A}$ is a set of $k\geq n$ points in $\PG(2n+1,q^{2})$ such that no $n+1$ points of $\mathcal{A}$ are contained in an $(n-1)$-space. We make two remarks about these semi-arcs. First, if $|S|>\binom{k}{n}\theta_{n-1}(q^{2})$, then $S$ contains a semi-arc with $k+1$ points, since it is possible to construct the semi-arc point by point: we start with a set of $n$ linearly independent points in $S$ and we extend the semi-arc point by point until we have $k+1$ points, which is possible by the condition on $S$. Secondly, if we choose $K$ points $\{P_{1},\dots,P_{K}\}$ in a semi-arc $\mathcal{A}\subseteq S$, then
  \begin{multline}\label{eq1}
    \sum_{\{i\}\in S_{K,1}}|P^{\sigma}_{i}\cap S|-\sum_{\{i,j\}\in S_{K,2}}|P^{\sigma}_{i}\cap P^{\sigma}_{j}\cap S|+\dots  \\+\sum_{\{i_{1},\dots,i_{2l+1}\}\in S_{K,2l+1}}|P^{\sigma}_{i_{1}}\cap P^{\sigma}_{i_{2}}\cap\dots\cap P^{\sigma}_{i_{2l+1}}\cap S|\geq|(P^{\sigma}_{1}\cup P^{\sigma}_{2}\cup\dots\cup P^{\sigma}_{K})\cap S|\, ,
  \end{multline}
  since every point of $(P^{\sigma}_{1}\cup P^{\sigma}_{2}\cup\dots\cup P^{\sigma}_{K})\cap S$ is counted at least once on the left hand side. Also
  \begin{multline}\label{eq2}
    \sum_{\{i\}\in S_{K,1}}|P^{\sigma}_{i}\cap S|-\sum_{\{i,j\}\in S_{K,2}}|P^{\sigma}_{i}\cap P^{\sigma}_{j}\cap S|+\dots  \\-\sum_{\{i_{1},\dots,i_{2l}\}\in S_{K,2l}}|P^{\sigma}_{i_{1}}\cap P^{\sigma}_{i_{2}}\cap\dots\cap P^{\sigma}_{i_{2l}}\cap S|\leq|(P^{\sigma}_{1}\cup P^{\sigma}_{2}\cup\dots\cup P^{\sigma}_{K})\cap S|\, ,
  \end{multline}
  since every point of $(P^{\sigma}_{1}\cup P^{\sigma}_{2}\cup\dots\cup P^{\sigma}_{K})\cap S$ is counted at most once on the left hand side. In both expressions we denoted the set of all subsets of $\{1,\dots,K\}$ of size $j$ by $S_{K,j}$.
  \par Now, we prove using induction, for every $0\leq t\leq n$, that we can find for any $(t+1)$-tuple $(c_{0},\dots,c_{t})$, $c_{j}>0$ a constant (independent of $q$), a constant $K_{t}\in\N$ such that 
  \[
    \forall K\geq K_{t}, \forall \{P_{1},\dots, P_{K}\}\subseteq\mathcal{A}\subseteq S: \sum_{\{i_{0},\dots,i_{t}\}\in S_{K,t+1}}|P^{\sigma}_{i_{0}}\cap P^{\sigma}_{i_{1}}\cap\dots\cap P^{\sigma}_{i_{t}}\cap S|\geq c_{t}q^{2n-1}.
  \]  
  We consider the case $t=0$, the induction base. Let $\{P_{1},\dots, P_{K}\}$ be a set of points in $\mathcal{A}\subseteq S$ (without restriction on $K$). By Theorem \ref{pperpcapS}, we know
  \[
    \sum^{K}_{i=1}|P^{\sigma}_{i}\cap S|\geq Kq^{2n-1}\,.
  \]
  Hence, it is sufficient to choose $K_{0}=\left\lceil c_{0}\right\rceil$.
  \par Next, we prove the induction step. We distinguish between two cases: $t$ even and $t$ odd. We look at the former, so we assume the inequality to be proven for $t\leq 2l-1$ and we prove it for $t=2l$. Let $K_{m}$ be the constant arising from the $(m+1)$-tuple $(c_{0},\ldots,c_{m})$, $m<2l$, and let $\{P_{1},\dots, P_{K}\}$ be a set of points in $\mathcal{A}\subseteq S$ with $K\geq K_{2l-1}$. By \eqref{eq1}, we know that
  \begin{multline*}
    \sum_{\{i\}\in S_{K,1}}|P^{\sigma}_{i}\cap S|-\sum_{\{i,j\}\in S_{K,2}}|P^{\sigma}_{i}\cap P^{\sigma}_{j}\cap S|+\dots  \\+\sum_{\{i_{0},\dots,i_{2l}\}\in S_{K,2l+1}}|P^{\sigma}_{i_{0}}\cap P^{\sigma}_{i_{1}}\cap\dots\cap P^{\sigma}_{i_{2l}}\cap S|\geq|(P^{\sigma}_{1}\cup P^{\sigma}_{2}\cup\dots\cup P^{\sigma}_{K})\cap S|\, .
  \end{multline*}
  Using the induction hypothesis and Theorem \ref{pperpcapS}, we find
  \begin{align*}
    \sum_{\{i_{0},\dots,i_{2l}\}\in S_{K,2l+1}}|P^{\sigma}_{i_{0}}\cap P^{\sigma}_{i_{1}}\cap\dots\cap P^{\sigma}_{i_{2l}}\cap S|&\geq \frac{\binom{K}{K_{2l-1}}}{\binom{K-2l}{K_{2l-1}-2l}}c_{2l-1}q^{2n-1}+\frac{\binom{K}{K_{2l-3}}}{\binom{K-2l+2}{K_{2l-3}-2l+2}}c_{2l-3}q^{2n-1}\\&\qquad\qquad+\dots+\frac{\binom{K}{K_{1}}}{\binom{K-2}{K_{1}-2}}c_{1}q^{2n-1}\\ &\quad-\left[\binom{K}{2l-1}+\binom{K}{2l-3}+\dots+K\right]\delta q^{2n-1}\\&\quad+q^{2n-1}
  \end{align*}
  and thus
  \begin{multline*}
    \sum_{\{i_{0},\dots,i_{2l}\}\in S_{K,2l+1}}|P^{\sigma}_{i_{0}}\cap P^{\sigma}_{i_{1}}\cap\dots\cap P^{\sigma}_{i_{2l}}\cap S| \\
    \geq \frac{\binom{K}{2l}}{\binom{K_{2l-1}}{2l}}c_{2l-1}q^{2n-1}+\frac{\binom{K}{2l-2}}{\binom{K_{2l-3}}{2l-2}}c_{2l-3}q^{2n-1}+\dots+\frac{\binom{K}{2}}{\binom{K_{1}}{2}}c_{1}q^{2n-1} \\
    -\left[\binom{K}{2l-1}+\binom{K}{2l-3}+\dots+K\right]\delta q^{2n-1}+q^{2n-1}\\=q^{2n-1}f(K,\delta,l,K_{1},K_{3},\ldots,K_{2l-1},c_{1},c_{3},\ldots,c_{2l-1}).
  \end{multline*}
  Note that $\frac{\binom{K}{K_{2i-1}}}{\binom{K-2i}{K_{2i-1}-2i}}=\frac{\binom{K}{2i}}{\binom{K_{2i-1}}{2i}}$. We now study the function $f$, which is clearly independent of $q$. Considering $f$ as a function of $K$ and comparing the exponents, we see that the term $\frac{\binom{K}{2l}}{\binom{K_{2l-1}}{2l}}c_{2l-1}$ dominates the others. Hence, we can find a value $K_{2l}\geq K_{2l-1}$ such that the right hand side is at least $c_{2l}q^{2n-1}$ for all $K\geq K_{2l}$, with $c_{2l}$ as chosen above. Then the statement follows. Note that $K_{2l}$ depends on the parameters $l,c_{1},\ldots,c_{2l}$ chosen before (the values $K_{i}$, $0\leq i<2l$, depend themselves on $i,c_{1},\ldots,c_{i}$).
  \par For the latter case, $t$ odd, the argument is similar, in this case starting from \eqref{eq2}.
  \par We will now apply the previous result for $t=n$. In order to do this, we need a semi-arc containing at least $K_{n}$ points. We argued in the beginning of the proof that $\delta q^{2n-1}=|S|>\binom{K_{n}-1}{n}\theta_{n-1}(q^{2})$ is a sufficient condition. Since $K_{n}$ is a constant, independent of $q$, and $\theta_{n-1}(q^{2})=q^{2n-2}+q^{2n-4}+\dots+q^{2}+1$, we can find $Q'_{1}>0$ such that this inequality is true for all $q\geq Q'_{1}$. Then we know
  \[
    \sum_{\{i_{0},\dots,i_{n}\}\in S_{K_{n},n+1}}|P^{\sigma}_{i_{0}}\cap P^{\sigma}_{i_{1}}\cap\dots\cap P^{\sigma}_{i_{n}}\cap S|\geq c_{n}q^{2n-1}
  \]
  for the points $\{P_{1},P_{2},\dots,P_{K_{n}}\}$ defining a semi-arc in $S$. Hence, we can find $n+1$ points - without loss of generality the points $\{P_{1},\dots,P_{n+1}\}$ - such that
  \[
    |P^{\sigma}_{1}\cap P^{\sigma}_{2}\cap\dots\cap P^{\sigma}_{n+1}\cap S|\geq\frac{c_{n}}{\binom{K_{n}}{n+1}}q^{2n-1}.
  \]
  We can find a constant $\overline{K}>0$ and a value $Q'\geq Q'_{1}$ such that $\frac{c_{n}}{\binom{K_{n}}{n+1}}q^{2n-1}\geq\overline{K}q^{2n-1}+\theta_{n-2}(q^{2})$ for $q\geq Q'$. We write $C_{n}=\overline{K}-\epsilon$, 
  $\max\{0,\overline{K}-\frac{2}{p}\}<\epsilon<\overline{K}$, and we denote the $n$-space $P^{\sigma}_{1}\cap P^{\sigma}_{2}\cap\dots\cap P^{\sigma}_{n+1}$ by $\pi$. Note that $\pi$ is an $n$-space since the points $P_{1},P_{2},\dots,P_{n+1}$ belong to a semi-arc. Then $|\pi\cap S|> C_{n} q^{2n-1}+\theta_{n-2}(q^{2})$.
  \par We know the intersection $\pi\cap H$ can be written as $\pi_{i}H_{n-i-1}$, with $H_{n-i-1}\cong\mathcal{H}(n-i-1,q^{2})$ and $\pi_{i}$ an $i$-space, $-1\leq i\leq n$. Let $Q''\geq Q'$ be such that $C_{n}q^{2n-1}+\theta_{n-2}(q^{2})>\delta\theta_{n-1}(q^{2})$ for all $q\geq Q''$. Such a value exists since the first term on the left hand side dominates the right hand side. If $i\geq n-1$, then $\pi\cap H$ is contained in a generator of $H$. Thus, using Lemma \ref{isotrope} and the assumption $q\geq Q''$ we find a contradiction. Hence, $i\leq n-2$. We find:
  \[
    |(\pi\Delta\pi^{\sigma})\cap S|\geq|(\pi\setminus\pi_{i})\cap S|\geq C_{n}q^{2n-1}+\theta_{n-2}(q^{2})-\theta_{i}(q^{2})\geq C_{n}q^{2n-1}\;.
  \]
  \par We still need to check the second claim in the statement of the lemma: $\frac{p-1}{p}|(\pi\Delta\pi^{\sigma})\cap H|<\Sigma_{n,i}-C_{n}q^{2n-1}$. Looking at the terms of highest degree in $\Sigma_{n,i}-C_{n}q^{2n-1}-\frac{p-1}{p}|(\pi\Delta\pi^{\sigma})\cap H|$, we find $2-C_{n}-2\frac{p-1}{p}=\epsilon-\frac{c_{n}}{\binom{K_{n}}{n+1}}+\frac{2}{p}>0$. Hence, we can find $Q\geq Q''$ such that the inequality $\frac{p-1}{p}|(\pi\Delta\pi^{\sigma})\cap H|<\Sigma_{n,i}-C_{n}q^{2n-1}$ holds for all $q\geq Q$.
\end{proof}

In this proof $\frac{c_{n}}{\binom{K_{n}}{n+1}}$ depends also on the choice of $c_{0},\dots,c_{n-1}$. So, investigating the possible values for $c_{0},\dots,c_{n}$, we can find many different values for $C_{n}$. With each of these values, a value $Q$ corresponds. We pick one of the possible values for $C_{n}$. By investigating different possibilities for $C_{n}$, we can see there is a trade-off between the choice of $C_{n}$ and the corresponding value $Q$.
\par From now on, we consider $C_{n}$ and the corresponding value $Q$ to be fixed. 

\begin{lemma}\label{baselemma}
  Let $c\in C_{n}(\mathcal{H}(2n+1,q^{2}))^{\bot}$ be a code word with $\wt(c)\leq w=\delta q^{2n-1}$, $\delta>0$ a constant, and denote $\supp(c)=S$. Consider $H\cong\mathcal{H}(2n+1,q^{2})$. Let $\pi$ be an $n$-space such that $\pi\cap H$ is a cone $\pi_{i}H_{n-i-1}$ with $H_{n-i-1}\cong\mathcal{H}(n-i-1,q^{2})$. Assume that $|S\cap(\pi\setminus\pi_{i})|=x$ and $|S\cap(\pi^{\sigma}\setminus\pi_{i})|=t$. Then there exists a value $Q_{n,i}\geq 0$ such that $x+t\leq C_{n}q^{2n-1}$ or $x+t\geq \Sigma_{n,i}-C_{n}q^{2n-1}$ if $q\geq Q_{n,i}$.
\end{lemma}
\begin{proof}
  Let $P$ be a point of $S\cap (\pi\setminus\pi_{i})$ and let $P'$ be a point of $((\pi^{\sigma}\cap H)\setminus\pi_{i})\backslash S$ and denote $\ell=PP'$. By Lemma \ref{formuleN} we know the number $N(n,i)$ of generators through $\ell$ intersecting $\pi$ and $\pi^{\sigma}$ in precisely one point, namely $P$ and $P'$. Each of these generators contains an additional point of $S$. Let $R$ be a point of $H\backslash(\pi\cup\pi^{\sigma})$. By Lemma \ref{generatorsthroughpoint} we know the number $n_{R}(n,i)$ of generators through $R$ intersecting both $\pi$ and $\pi^{\sigma}$ in a point. Hence, $S\setminus(\pi\cup\pi^{\sigma})$ contains at least
  \[
    x(|(\pi^{\sigma}\cap H)\setminus\pi_{i}|-t)\frac{N(n,i)}{n_{\max}(n,i)}=x(q^{2i+2}\mu_{n-i-1}(q^{2})-t)\frac{N(n,i)}{n_{\max}(n,i)}
  \]
  points, whereby $n_{\max}(n,i)=\max_{R\in S\setminus(\pi\cup\pi^{\sigma})}n_{R}(n,i)$. Switching the roles of $\pi$ and $\pi^{\sigma}$, and adding these two inequalities, we find after dividing by two
  \[
    x(q^{2i+2}\mu_{n-i-1}(q^{2})-t)\frac{N(n,i)}{2n_{\max}(n,i)}+t(q^{2i+2}\mu_{n-i-1}(q^{2})-x)\frac{N(n,i)}{2n_{\max}(n,i)}+x+t\leq |S|\leq w\;.
  \]
  Rewriting this inequality yields
  \[
    (x+t)\left(q^{2i+2}\mu_{n-i-1}(q^{2})N(n,i)+2n_{\max}(n,i)\right)-2xtN(n,i)\leq 2w\,n_{\max}(n,i)\;.
  \]
  Using the inequality $2xt\leq\frac{1}{2}(x+t)^{2}$ and writing $y=x+t$, we find
  \[
    \frac{1}{2}y^{2}N(n,i)-\left[q^{2i+2}\mu_{n-i-1}(q^{2})N(n,i)+2n_{\max}(n,i)\right]y+2w\,n_{\max}(n,i)\geq0\;.
  \]
  \par We now distinguish between two cases: $n-i$ odd and $n-i$ even. First we look at the former. By detailed analysis one can see that in this case
  \begin{multline*}
    N(n-1,i)q^{4i+4}\left(\mu_{n-i-2}(q^{2})\right)^{2}\\ \geq N(n-1,i-1)q^{4i}\left(\mu_{n-i-1}(q^{2})\right)^{2}\\ \geq N(n-1,i+1)q^{4i+4}\mu_{n-i-3}(q^{2})\left[q^{4}\mu_{n-i-3}(q^{2})+q^{2}-1\right]
  \end{multline*}
  if $n-i>3$ and
  \[
    N(n-1,n-3)q^{4n-8}\left(q+1\right)^{2}\\ \geq N(n-1,n-4)q^{4n-12}\left(q^{3}+1\right)^{2}\\ \geq N(n,n-3)q^{2n-4}\,.
  \]
  These inequalities correspond to $i=n-3$. Hence, $n_{\max}(n,i)=N(n-1,i)q^{4i+4}\left(\mu_{n-i-2}(q^{2})\right)^{2}$. Using the formula for $N(n,i)$ from Lemma \ref{formuleN}, and simplifying, we can rewrite this inequality as
  \begin{multline}\label{eqodd}
    \frac{1}{2}q^{n-3i-5}y^2-\left[q^{n-i-3}\mu_{n-i-1}(q^{2})+2\mu_{n-i-2}(q^{2})\frac{q^{n-i-1}-1}{q^{2}-1}\right]y\\+2\delta q^{2n-1}\mu_{n-i-2}(q^{2})\frac{q^{n-i-1}-1}{q^{2}-1}\geq0\;.
  \end{multline}
  Let $\alpha_{n,i}(q^{2})$ and $\alpha'_{n,i}(q^{2})$ be the two solutions of the corresponding equation, with $\alpha_{n,i}(q^{2})\leq \alpha'_{n,i}(q^{2})$. Then $x+t\leq \alpha_{n,i}(q^{2})$ or $x+t\geq \alpha'_{n,i}(q^{2})$. Moreover,
  \[
    \alpha_{n,i}(q^{2})+\alpha'_{n,i}(q^{2})=2q^{2i+2}\mu_{n-i-1}(q^{2})+4\frac{\mu_{n-i-2}(q^{2})(q^{n-i-1}-1)}{q^{n-3i-5}(q^{2}-1)}=\Sigma_{n,i}\;.
  \]
  For the given $\delta$ we calculate 
  \begin{align*}
    \overline{\alpha_{n,i}}=\lim_{q\to\infty}\alpha_{n,i}(q^{2})=\lim_{q\to\infty}\frac{B'-\sqrt{B'^{2}-4\delta q^{3n-3i-6}C'}}{q^{n-3i-5}}\; ,
  \end{align*}
  with
  \begin{align*}
    B'&=q^{n-i-3}\mu_{n-i-1}(q^{2})+2\mu_{n-i-2}(q^{2})\frac{q^{n-i-1}-1}{q^{2}-1}\;,\\
    C'&=\mu_{n-i-2}(q^{2})\frac{q^{n-i-1}-1}{q^{2}-1}\;.
  \end{align*}
  Since $\overline{\alpha_{n,i}}\in O(q^{2n-2})$, we can find $Q_{n,i}>0$ such that $\alpha_{n,i}(q^{2})\leq C_{n}q^{2n-1}$ for $q\geq Q_{n,i}$.
  \par In the latter case, $n-i$ even, similar arguments can be used. However, in this case we need to distinguish between $n-i>2$ and $i=n-2$. First, we discuss $n-i>2$. We can deduce that
  \begin{multline*}
    N(n-1,i)q^{4i+4}\left(\mu_{n-i-2}(q^{2})\right)^{2}\\ \leq N(n-1,i-1)q^{4i}\left(\mu_{n-i-1}(q^{2})\right)^{2}\\ \leq N(n-1,i+1)q^{4i+4}\mu_{n-i-3}(q^{2})\left[q^{4}\mu_{n-i-3}(q^{2})+q^{2}-1\right]\, ,
  \end{multline*}
  hence $n_{\max}(n,i)=N(n-1,i+1)q^{4i+4}\mu_{n-i-3}(q^{2})\left[q^{4}\mu_{n-i-3}(q^{2})+q^{2}-1\right]$. We find the inequality
  \begin{multline}\label{eqeven}
    \frac{q^{2}-1}{2}y^{2}-q^{2i+2}\left[\mu_{n-i-1}(q^{2})(q^{2}-1)+2(q^{4}\mu_{n-i-3}(q^{2})+q^{2}-1)\right]y\\ +2\delta q^{2n-1}q^{2i+2}(q^{4}\mu_{n-i-3}(q^{2})+q^{2}-1)\geq0\;.
  \end{multline}
  Just as in the previous case, we define $\Sigma_{n,i}$, which is the sum of the solutions of the corresponding equation, and $\overline{\alpha_{n,i}}$:
  \begin{align*}
    \Sigma_{n,i}&=2q^{2i+2}\left[\mu_{n-i-1}(q^{2})+2\frac{q^{4}\mu_{n-i-3}(q^{2})+q^{2}-1}{q^{2}-1}\right]\;,\\
    \overline{\alpha_{n,i}}&=\lim_{q\to\infty}\frac{B''-\sqrt{B''^{2}-4\delta q^{2n-1}(q^{2}-1)C''}}{q^{2}-1}\; ,
  \end{align*}
  with
  \begin{align*}
    B''&=q^{2i+2}\left[\mu_{n-i-1}(q^{2})(q^{2}-1)+2(q^{4}\mu_{n-i-3}(q^{2})+q^{2}-1)\right]\; ,\\
    C''&=q^{2i+2}(q^{4}\mu_{n-i-3}(q^{2})+q^{2}-1)\; .
  \end{align*}
  Since $\overline{\alpha_{n,i}}\in O(q^{2n-2})$ also holds in this case, we again can find $Q_{n,i}>0$ such that $\alpha_{n,i}(q^{2})\leq C_{n}q^{2n-1}$ for $q\geq Q_{n,i}$.
  \par Finally, we consider the case $i=n-2$. The second possibility in Lemma \ref{generatorsthroughpoint} can thus not occur. We note that
  \[
    N(n-1,n-3)q^{4(n-2)}(q+1)^{2}\leq q^{2n-2}N(n,n-2)\,.
  \]
  The arguments in this case are analogous.
  \par Hence, in all cases we can find $Q_{n,i}>0$ such that $x+t\leq C_{n}q^{2n-1}$ or $x+t\geq\Sigma_{n,i}-C_{n}q^{2n-1}$ for $q\geq Q_{n,i}$. 
\end{proof}

Using the three previous lemmas, we can now prove a classification theorem for the small weight code words in $C_{n}(\mathcal{H}(2n+1,q^{2}))^{\bot}$.

\begin{theorem}\label{maintheorema}
  Let $p$ be a fixed prime, $\delta>0$ be a fixed constant and $n$ be a fixed positive integer. Then there is a constant $\overline{Q}$ such that, for any $q=p^{h}$ with $h\in\N$ and $q\ge\overline{Q}$, and any $c\in C_{n}(\mathcal{H}(2n+1,q^{2}))^{\bot}$ with $\wt(c)\leq w=\delta q^{2n-1}$, $c$ is a linear combination of code words described in Theorem \ref{codewords}.
\end{theorem}
\begin{proof}
  For the given values $p$ and $\delta$ we have found a set of possible $C_{n}$-values, of which we have chosen one, in Lemma \ref{subspace},  with $Q$, a power of $p$, corresponding to it. By the proof of this lemma, we know that $C_{n}q^{2n-1}>\delta\theta_{n-1}(q^{2})$ for all $q\geq Q$. Define $\overline{Q}=\max(\{Q\}\cup\{Q_{n,i}\mid -1\leq i\leq n-2\})$, with $Q_{n,i}$ as in Lemma \ref{baselemma}, corresponding to the chosen value $C_{n}$. We assume $q\geq\overline{Q}$.
  \par  Denote $\supp(c)=S$. By Lemma \ref{subspace}, we find an $n$-space $\pi$ such that $N:=|(\pi\Delta\pi^{\sigma})\cap S|> C_{n} q^{2n-1}$. The intersection $\pi\cap H$ can be written as $\pi_{i}H_{n-i-1}$, with $H_{n-i-1}\cong\mathcal{H}(n-i-1,q^{2})$, $-1\leq i\leq n-2$.
  \par Since $N> C_{n} q^{2n-1}$ and $q\geq Q_{n,i}$, we know by Lemma \ref{baselemma} that $N\geq \Sigma_{n,i}-C_{n}q^{2n-1}$. For each element $\alpha\in\F^{*}_{p}$, we denote by $N_{\alpha}$ the sum of the number of points $P\in\pi$ such that $c_{P}=\alpha$ and the number of points $Q\in\pi^{\sigma}$ such that $c_{Q}=-\alpha$. We can find $\beta\in\F^{*}_{p}$ such that $N_{\beta}\geq \frac{N}{p-1}$. We now consider the code word $c'=c-\beta(v_{\pi}-v_{\pi^{\sigma}})$, with $v_{\pi}$ and $v_{\pi^{\sigma}}$ as in Theorem \ref{codewords}. We know
  \[
    \wt(c')=(N-N_{\beta})+(|(\pi\Delta\pi^{\sigma})\cap H|-N)=|(\pi\Delta\pi^{\sigma})\cap H|-N_{\beta}\leq |(\pi\Delta\pi^{\sigma})\cap H|-\frac{N}{p-1}\;.
  \]
  We also know that $N\geq \Sigma_{n,i}-C_{n}q^{2n-1}>\frac{p-1}{p}|(\pi\Delta\pi^{\sigma})\cap H|$ by Lemma \ref{subspace}. It follows that
  \[
    \wt(c')<\frac{p}{p-1}N-\frac{N}{p-1}=N\leq\wt(c)\;.
  \]
  Hence, the theorem follows using induction on $w=\wt(c)$.
\end{proof}

We now focus on the code words that we described in Section \ref{sec:codewords}.

\begin{remark}\label{maintheoremb}
  Let $c$ be a small weight code word and $q$ sufficiently large. Following the arguments in the proof of Theorem \ref{maintheorema}, we know that $c=c_{1}+\dots+c_{m}$, with $c_{i}$, $1\leq i\leq m$, a code word that we described in Theorem \ref{codewords} and Example \ref{possibilities}, such that $\wt(c_{1}+\dots+c_{m'})<\wt(c_{1}+\dots+c_{m'+1})$ for all $1\leq m'\leq m$. From this observation, it immediately follows that the code words that we described in Theorem \ref{codewords} and Example \ref{possibilities} are the code words of smallest weights.
  \par Now we consider small weight code words different from the ones described in Theorem \ref{codewords}. Let $c$ be a code word $c$ of weight smaller than $4q^{2n-1}$, $q$ sufficiently large. Since $c$ is not of the type we described in Theorem \ref{codewords}, $c$ can be written as a linear combination of at least two of these code words. By the above arguments, we can find a code word $c'$ which is a linear combination of precisely two of these code words, such that $\wt(c')\leq\wt(c)$. In particular, we can find $\alpha,\alpha'\in\F^{*}_{p}$ and $n$-spaces $\pi,\pi'$, $\pi\notin\{\pi',\pi'^{\sigma}\}$, such that $c'=\alpha(v_{\pi}-v_{\pi^{\sigma}})+\alpha'(v_{\pi'}-v_{\pi'^{\sigma}})$ and $\wt(c')<4q^{2n-1}$. Let $S$ be the support of $c'$. We know $S\subseteq((\pi\Delta\pi^{\sigma})\cup(\pi'\Delta\pi'^{\sigma}))\cap\mathcal{H}(2n+1,q^{2})$. However, it can be seen that $|(\pi\Delta\pi^{\sigma})\cap(\pi'\Delta\pi'^{\sigma})|\leq 4q^{2n-2}$. Hence,
  \[
    |S|\geq\wt(\alpha(v_{\pi}-v_{\pi^{\sigma}}))+\wt(\alpha'(v_{\pi'}-v_{\pi'^{\sigma}}))-|(\pi\Delta\pi^{\sigma})\cap(\pi'\Delta\pi'^{\sigma})|\geq 4q^{2n-1}\;,
  \]
  a contradiction. It follows that the only code words of weight smaller than $4q^{2n-1}$ are of the type described in Theorem \ref{codewords}.
\end{remark}

Note that Theorem \ref{maintheorema} only proves the second half of Theorem \ref{maintheorem}. From Remark \ref{maintheoremb} now the first half also follows.


\vspace{5mm} \textbf{Acknowledgment.} The authors wish to thank Leo Storme for his useful comments and suggestions on earlier drafts of the paper.

\end{document}